\documentclass{amsart}
\usepackage[pdfencoding=auto,hidelinks]{hyperref}
\usepackage{bm}
\usepackage{mathtools}
\newtheorem{thm}{Theorem}[section]
\newtheorem{lem}[thm]{Lemma}
\newtheorem{prop}[thm]{Proposition}
\newtheorem{cor}[thm]{Corollary}
\theoremstyle{definition}
\newtheorem{defn}[thm]{Definition}
\newtheorem{ex}[thm]{Example}
\theoremstyle{remark}
\newtheorem{rem}[thm]{Remark}
\DeclareMathOperator{\Li}{Li}
\DeclareMathOperator{\spf}{sf}
\newcommand{\Q}{\mathbb{Q}}
\newcommand{\R}{\mathbb{R}}
\newcommand{\Z}{\mathbb{Z}}
\newcommand{\lau}{\mathbb{Q}(\!(z)\!)}
\newcommand{\lauseq}{\lau^{<\omega}}
\title[A determinant formula for multiple harmonic sums]{A determinant formula for multiple harmonic sums and their linear independence}
\author{Minoru Hirose}
\address{Graduate School of Science and Engineering, Kagoshima University, 1-21-35 Korimoto, Kagoshima, Kagoshima 890-0065, Japan}
\email{hirose@sci.kagoshima-u.ac.jp}
\author{Hideki Murahara}
\address{Department of Economics, The University of Kitakyushu, 4-2-1 Kitagata, Kokuraminami-ku, Kitakyushu, Fukuoka, 802-8577, Japan}
\email{hmurahara@mathformula.page}
\author{Shingo Saito}
\address{Faculty of Arts and Science, Kyushu University, 744, Motooka, Nishi-ku, Fukuoka, 819-0395, Japan}
\email{ssaito@artsci.kyushu-u.ac.jp}
\begin{document}
\begin{abstract}
 We explicitly evaluate a family of determinants whose entries are multiple harmonic sums.
 As a consequence, we prove that certain finite sequences of multiple harmonic sums are linearly independent.
 This result strengthens the previously known linear independence of the corresponding infinite sequences
 and stands in contrast to the abundance of linear relations among multiple zeta values.
\end{abstract}
\maketitle

\section{Introduction}
An \emph{index} is a finite (possibly empty) sequence of positive integers.
For each index $\bm{k}=(k_1,\dots,k_r)$ and each $n\ge1$, we define the \emph{multiple harmonic sum} $\zeta_{<n}(\bm{k})$ by
\[
 \zeta_{<n}(\bm{k})=\sum_{1\le n_1<\dots<n_r<n}\frac{1}{n_1^{k_1}\dotsm n_r^{k_r}}\in\Q,
\]
where we understand that $\zeta_{<n}(\emptyset)=1$ for all $n$ and that $\zeta_{<n}(\bm{k})=0$ whenever $n\le r$.

The aim of this paper is to evaluate a determinant whose entries are multiple harmonic sums.
To state the main results, we introduce the following notation.
For each index $\bm{k}=(k_1,\dots,k_r)$,
we follow \cite{S} in writing $\bm{k}^{\uparrow}=(k_1,\dots,k_{r-1},k_r+1)$ and $\bm{k}^{\rightarrow}=(k_1,\dots,k_r,1)$,
where $\emptyset^{\uparrow}=\emptyset$ and $\emptyset^{\rightarrow}=(1)$.
There is a unique enumeration $\bm{k}_0,\bm{k}_1,\dotsc$ of all indices such that $\bm{k}_{2i}=\bm{k}_i^{\uparrow}$ and $\bm{k}_{2i+1}=\bm{k}_i^{\rightarrow}$
for all $i\ge0$.

\begin{ex}
 The first few indices in the enumeration are
 $\bm{k}_0=\emptyset$, $\bm{k}_1=(1)$, $\bm{k}_2=(2)$, $\bm{k}_3=(1,1)$, $\bm{k}_4=(3)$, $\bm{k}_5=(2,1)$, $\bm{k}_6=(1,2)$,
 and $\bm{k}_7=(1,1,1)$.
\end{ex}

Note that for each $w\ge1$, the indices $\bm{k}_0,\dots,\bm{k}_{2^w-1}$ are precisely the indices $\bm{k}=(k_1,\dots,k_r)$ with $k_1+\dots+k_r\le w$.

We are now ready to state our main theorem:
\begin{thm}[Main Theorem]\label{thm:<}
 Let $w\ge1$ and $N\ge0$.
 Then we have
 \begin{align*}
  &\begin{vmatrix}\zeta_{<N+1}(\bm{k}_0)&\cdots&\zeta_{<N+1}(\bm{k}_{2^w-1})\\\vdots&\ddots&\vdots\\\zeta_{<N+2^w}(\bm{k}_0)&\cdots&\zeta_{<N+2^w}(\bm{k}_{2^w-1})\end{vmatrix}\\
  &=(-1)^{1-\delta_{w,1}}\prod_{j=0}^{w-1}\prod_{i=\lfloor2^{j-1}\rfloor}^{2^j-1}\frac{(i!)^{2^{w-j+2}-2(w-j+2)}((N+i)!)^{w-j}}{((N+2^w-1-i)!)^{w-j}},
 \end{align*}
 where $\delta_{w,1}$ denotes the Kronecker delta.
 In particular, setting $N=0$, we have
 \[
  \begin{vmatrix}\zeta_{<1}(\bm{k}_0)&\cdots&\zeta_{<1}(\bm{k}_{2^w-1})\\\vdots&\ddots&\vdots\\\zeta_{<2^w}(\bm{k}_0)&\cdots&\zeta_{<2^w}(\bm{k}_{2^w-1})\end{vmatrix}
  =(-1)^{1-\delta_{w,1}}\prod_{j=0}^{w-1}\prod_{i=\lfloor2^{j-1}\rfloor}^{2^j-1}\frac{(i!)^{2^{w-j+2}-w+j-4}}{((2^w-1-i)!)^{w-j}}.
 \]
\end{thm}

\begin{rem}
 The right-hand side of the first equation in Theorem~\ref{thm:<} simplifies to the reciprocal of a polynomial in $N$,
 because $N+i\le N+2^w-1-i$ for all $i$ in the product.
\end{rem}

\begin{ex}
 For $w=1,2,3,4$, the right-hand side of the first equation in Theorem~\ref{thm:<} is equal to $c_w/p_w(N)$, where
 \begin{align*}
  c_1&=1,\\
  p_1(N)&=N+1,\\
  c_2&=-1,\\
  p_2(N)&=(N+1)^2(N+2)^3(N+3)^2,\\
  c_3&=-144,\\
  p_3(N)&=(N+1)^3(N+2)^5(N+3)^6(N+4)^7(N+5)^6(N+6)^5(N+7)^3,\\
  c_4&=-46963559907481570246656000000,\\
  p_4(N)&=(N+1)^4(N+2)^7(N+3)^9(N+4)^{11}(N+5)^{12}\\
  &\hphantom{{}={}}\times(N+6)^{13}(N+7)^{14}(N+8)^{15}(N+9)^{14}(N+10)^{13}\\
  &\hphantom{{}={}}\times(N+11)^{12}(N+12)^{11}(N+13)^{9}(N+14)^{7}(N+15)^{4},
 \end{align*}
 and therefore the right-hand side of the second equation in Theorem~\ref{thm:<} is equal to
 \begin{align*}
  \frac{c_1}{p_1(0)}&=1,&
  \frac{c_2}{p_2(0)}&=-\frac{1}{2^{3}3^{2}}=-\frac{1}{72},\\
  \frac{c_3}{p_3(0)}&=-\frac{1}{2^{20}3^{9}5^{6}7^{3}},&
  \frac{c_4}{p_4(0)}&=-\frac{1}{2^{85}3^{45}5^{23}7^{19}11^{12}13^{9}}.
 \end{align*}
\end{ex}

Since the right-hand side of the first equation in Theorem~\ref{thm:<} is not zero,
Theorem~\ref{thm:<} implies the following corollary:
\begin{cor}
 Let $w\ge1$ and $N\ge0$.
 Then the finite sequences
 \[
  (\zeta_{<n}(\bm{k}_i))_{N+1\le n\le N+2^w}\in\Q^{2^w}
 \]
 for $i=0,\dots,2^w-1$ are linearly independent.
 In particular, the finite sequences
 \[
  (\zeta_{<n}(\bm{k}_i))_{1\le n\le2^w}\in\Q^{2^w}
 \]
 for $i=0,\dots,2^w-1$ are linearly independent.
\end{cor}

Consequently, the infinite sequences $(\zeta_{<n}(\bm{k}))_{n\ge1}\in\Q^{\Z_{\ge1}}$ for indices $\bm{k}$ are linearly independent.
As pointed out by Yamamoto~\cite[Theorem~3.1]{Y}, this also follows from the linear independence of multiple polylogarithms
(see Definition~\ref{defn:multiple_polylog}),
established in \cite[Theorem~4]{MPH}, \cite[Theorem~3]{U}, and \cite[Theorem~8.7]{W} using complex analysis.
The linear independence of these sequences stands in stark contrast to the fact that the \emph{multiple zeta values} defined by
\[
 \zeta(\bm{k})=\lim_{n\to\infty}\zeta_{<n}(\bm{k})=\sum_{1\le n_1<\dots<n_r}\frac{1}{n_1^{k_1}\dotsm n_r^{k_r}}\in\R
\]
for admissible (i.e.\ $k_r\ge2$ or $r=0$) indices $\bm{k}=(k_1,\dots,k_r)$
satisfy numerous linear relations, one of the simplest being $\zeta(3)=\zeta(1,2)$
(see, for example, \cite{K} for the basics of multiple zeta values).

\section{Proof of the main theorem}
In this section, we prove Theorem~\ref{thm:<}.

\subsection{Slight modification of multiple harmonic sums}
For the proof of Theorem~\ref{thm:<}, it is convenient to work with a slightly modified version of the multiple harmonic sums.
For each nonempty index $\bm{k}=(k_1,\dots,k_r)$ and each $n\ge1$, we define
\[
 \zeta_{=n}(\bm{k})=\sum_{1\le n_1<\dots<n_r=n}\frac{1}{n_1^{k_1}\dotsm n_r^{k_r}}\in\Q,
\]
where we understand that $\zeta_{=n}(\bm{k})=0$ whenever $n<r$.
Since $\zeta_{<n}(\bm{k}_0)=1$ for all $n$ and $\zeta_{<n+1}(\bm{k})-\zeta_{<n}(\bm{k})=\zeta_{=n}(\bm{k})$ for all nonempty indices $\bm{k}$ and all $n$,
elementary row operations reduce Theorem~\ref{thm:<} to the following proposition:
\begin{prop}\label{prop:=}
 Let $w\ge1$ and $N\ge0$.
 Then we have
 \begin{align*}
  &\begin{vmatrix}\zeta_{=N+1}(\bm{k}_1)&\cdots&\zeta_{=N+1}(\bm{k}_{2^w-1})\\\vdots&\ddots&\vdots\\\zeta_{=N+2^w-1}(\bm{k}_1)&\cdots&\zeta_{=N+2^w-1}(\bm{k}_{2^w-1})\end{vmatrix}\\
  &=(-1)^{1-\delta_{w,1}}\prod_{j=0}^{w-1}\prod_{i=\lfloor2^{j-1}\rfloor}^{2^j-1}\frac{(i!)^{2^{w-j+2}-2(w-j+2)}((N+i)!)^{w-j}}{((N+2^w-1-i)!)^{w-j}}.
 \end{align*}
\end{prop}

\subsection{Multiple polylogarithms}
\begin{defn}\label{defn:multiple_polylog}
 Let $\bm{k}=(k_1,\dots,k_r)$ be a nonempty index.
 We define the \emph{multiple polylogarithm} $\Li_{\bm{k}}(z)$ by
 \[
  \Li_{\bm{k}}(z)=\sum_{1\le n_1<\dots<n_r}\frac{z^{n_r}}{n_1^{k_1}\dotsm n_r^{k_r}}\in\Q[[z]].
 \]
\end{defn}

For notational simplicity, for a formal Laurent series $f(z)\in\lau$, we suppress the argument $z$ and simply write $f$;
in particular, we write $\Li_{\bm{k}}$ for $\Li_{\bm{k}}(z)$.
For $f\in\lau$ and $n\in\Z$, we write $[z^n]f$ for the coefficient of $z^n$ in $f$.
Observe that
\[
 [z^n]\Li_{\bm{k}}=\zeta_{=n}(\bm{k})
\]
for all nonempty indices $\bm{k}$ and $n\ge1$,
so that the left-hand side of Proposition~\ref{prop:=} is equal to
\[
 \begin{vmatrix}[z^{N+1}]\Li_{\bm{k}_1}&\cdots&[z^{N+1}]\Li_{\bm{k}_{2^w-1}}\\\vdots&\ddots&\vdots\\ [z^{N+2^w-1}]\Li_{\bm{k}_1}&\cdots&[z^{N+2^w-1}]\Li_{\bm{k}_{2^w-1}}\end{vmatrix}.
\]

\subsection{Properties of $D_N$}
Let $\lauseq=\bigcup_{l\ge0}\lau^l$ denote the set of all finite (possibly empty) sequences of formal Laurent series.
For $N\in\Z$ and $\bm{f}=(f_1,\dots,f_l)\in\lauseq$, we write
\[
 D_N(\bm{f})=\begin{vmatrix}[z^{N+1}]f_1&\cdots&[z^{N+1}]f_l\\\vdots&\ddots&\vdots\\ [z^{N+l}]f_1&\cdots&[z^{N+l}]f_l\end{vmatrix},
\]
where we understand that $D_N(\emptyset)=1$ for all $N\in\Z$.
With this notation, the left-hand side of Proposition~\ref{prop:=} is equal to $D_N(\Li_{\bm{k}_1},\dots,\Li_{\bm{k}_{2^w-1}})$.

\begin{lem}\label{lem:1-z}
 Let $m\ge0$, $N\ge-1$, and $\bm{f}=(f_1,\dots,f_l)\in\lauseq$.
 Then we have
 \[
  D_N((1-z)^{-m},\dots,(1-z)^{-1},\bm{f})
  =(-1)^{m(m-1)/2}D_{N+m}((1-z)^m\bm{f}),
 \]
 where $(1-z)^m\bm{f}=((1-z)^mf_1,\dots,(1-z)^mf_l)$.
\end{lem}

\begin{proof}
 The proof proceeds by induction on $m$.
 For $m=0$, the desired formula holds trivially since both sides coincide with $D_N(\bm{f})$.
 For $m=1$, since $[z^{N+i}](1-z)^{-1}=1$ for $i=1,\dots,l+1$, performing elementary row operations gives
 \begin{align*}
  D_N((1-z)^{-1},\bm{f})
  &=\begin{vmatrix}1&[z^{N+1}]f_1&\cdots&[z^{N+1}]f_l\\\vdots&\vdots&\ddots&\vdots\\1&[z^{N+l+1}]f_1&\cdots&[z^{N+l+1}]f_l\end{vmatrix}\\
  &=\begin{vmatrix}0&[z^{N+1}]f_1-[z^{N+2}]f_1&\cdots&[z^{N+1}]f_l-[z^{N+2}]f_l\\\vdots&\vdots&\ddots&\vdots\\0&[z^{N+l}]f_1-[z^{N+l+1}]f_1&\cdots&[z^{N+l}]f_l-[z^{N+l+1}]f_l\\1&[z^{N+l+1}]f_1&\cdots&[z^{N+l+1}]f_l\end{vmatrix}\\
  &=(-1)^l\begin{vmatrix}-[z^{N+2}](1-z)f_1&\cdots&-[z^{N+2}](1-z)f_l\\\vdots&\ddots&\vdots\\-[z^{N+l+1}](1-z)f_1&\cdots&-[z^{N+l+1}](1-z)f_l\end{vmatrix}\\
  &=D_{N+1}((1-z)\bm{f}).
 \end{align*}
 If the formula holds for $m$, then using the $m=1$ case and moving the $(m+1)$st column to the first position give
 \begin{align*}
  &D_N((1-z)^{-m-1},\dots,(1-z)^{-1},\bm{f})\\
  &=(-1)^mD_N((1-z)^{-1},(1-z)^{-m-1},\dots,(1-z)^{-2},\bm{f})\\
  &=(-1)^mD_{N+1}((1-z)^{-m},\dots,(1-z)^{-1},(1-z)\bm{f})\\
  &=(-1)^m(-1)^{m(m-1)/2}D_{N+m+1}((1-z)^{m+1}\bm{f})\\
  &=(-1)^{m(m+1)/2}D_{N+m+1}((1-z)^{m+1}\bm{f}),
 \end{align*}
 which proves that the formula holds for $m+1$.
\end{proof}

For $\bm{f}=(f_1,\dots,f_l)\in\lauseq$, we write $\bm{f}^{(m)}=(f_1^{(m)},\dots,f_l^{(m)})\in\lauseq$ for the sequence of the $m$-th derivatives of the entries of $\bm{f}$.
For $n\ge-1$ and $m\ge0$, we write $\spf_n(m)=\prod_{i=n+1}^{n+m}i!$, where we understand that $\spf_n(0)=1$.
Note that $\spf_0(m)=\prod_{i=1}^{m}i!$ is known as the superfactorial of $m$.

\begin{lem}\label{lem:diff}
 Let $m\ge0$, $N\ge m-1$, and $\bm{f}=(f_1,\dots,f_l)\in\lauseq$.
 Then we have
 \[
  D_N(\bm{f})=\frac{\spf_{N-m}(m)}{\spf_{N+l-m}(m)}D_{N-m}(\bm{f}^{(m)}).
 \]
\end{lem}

\begin{proof}
 The proof proceeds by induction on $m$.
 The formula is trivial for $m=0$.
 For $m=1$, we have $N\ge0$ and so
 \begin{align*}
  D_N(\bm{f})
  &=\begin{vmatrix}[z^{N+1}]f_1&\cdots&[z^{N+1}]f_l\\\vdots&\ddots&\vdots\\ [z^{N+l}]f_1&\cdots&[z^{N+l}]f_l\end{vmatrix}\\
  &=\begin{vmatrix}\frac{1}{N+1}[z^{N}]f_1'&\cdots&\frac{1}{N+1}[z^{N}]f_l'\\\vdots&\ddots&\vdots\\\frac{1}{N+l}[z^{N+l-1}]f_1'&\cdots&\frac{1}{N+l}[z^{N+l-1}]f_l'\end{vmatrix}
  =\frac{N!}{(N+l)!}D_{N-1}(\bm{f}').
 \end{align*}
 If the formula holds for $m$, then for $N\ge m$, using the $m=1$ case gives
 \begin{align*}
  D_N(\bm{f})
  &=\frac{\spf_{N-m}(m)}{\spf_{N+l-m}(m)}D_{N-m}(\bm{f}^{(m)})\\
  &=\frac{\spf_{N-m}(m)}{\spf_{N+l-m}(m)}\cdot\frac{(N-m)!}{(N+l-m)!}D_{N-m-1}(\bm{f}^{(m+1)})\\
  &=\frac{\spf_{N-m-1}(m+1)}{\spf_{N+l-m-1}(m+1)}D_{N-m-1}(\bm{f}^{(m+1)}),
 \end{align*}
 which proves that the formula holds for $m+1$.
\end{proof}

\subsection{Computation of derivatives of multiple polylogarithms}
For simplicity, we write $L_i=\Li_{\bm{k}_i}$ for $i\ge1$ and set $L_0=1$.
Then we have
\[
 L_i'=
 \begin{cases}
  (1-z)^{-1}L_{(i-1)/2}&\text{if $i$ is odd};\\
  z^{-1}L_{i/2}&\text{if $i$ is even and $i\ge2$}.
 \end{cases}
\]
In addition, for $p\ge0$ and $i\ge1$, we write $L_{p,i}=z^pL_i\in\Q[[z]]$.

For $p,q\ge0$ and $i\ge1$, we define $\Q$-linear subspaces $V_{p,i}^{(q)}$ and $\tilde{V}_{p,i}^{(q)}$ of $\lau$ by
\begin{align*}
 V_{p,i}^{(q)}&=\langle L_{p',i'}^{(q)}\mid0\le p'\le\max\{p,q-1\},\;1\le i'<i\rangle+\langle L_{p',i}^{(q)}\mid0\le p'<p\rangle\\
 &\hphantom{{}={}}+\langle z^r\mid0\le r\le p-q\rangle+\langle z^{-r}\mid1\le r\le q-p-1\rangle
\end{align*}
and $\tilde{V}_{p,i}^{(q)}=V_{p,i}^{(q)}+\langle L_{p,i}^{(q)}\rangle$,
where $\langle\cdot\rangle$ denotes the $\Q$-linear span.

\begin{lem}\label{lem:Vpiq_Vpi(q+1)}
 For $p\ge0$ and $q,i\ge1$, we have $\{f'\mid f\in V_{p,i}^{(q-1)}\}\subset V_{p,i}^{(q)}$.
\end{lem}

\begin{proof}
 The lemma can be verified by a straightforward case analysis.
\end{proof}

\begin{prop}\label{prop:Lpi(q)}
 Let $p,q\ge0$ and $i\ge1$.
 Then we have
 \[
  L_{p,i}^{(q)}\equiv
  \begin{cases}
   p!(q-p-1)!(1-z)^{-(q-p)}L_{(i-1)/2}&\text{if $p<q$ and $i$ is odd;}\\
   (-1)^{q-p-1}p!(q-p-1)!z^{-(q-p)}L_{i/2}&\text{if $p<q$ and $i$ is even;}\\
   \frac{p!}{(p-q)!}L_{p-q,i}&\text{if $p\ge q$}
  \end{cases}
  \pmod{V_{p,i}^{(q)}}.
 \]
\end{prop}

The remainder of this subsection is devoted to proving Proposition~\ref{prop:Lpi(q)}.
For $p,q\ge0$ and $i\ge1$,
let $P(p,q,i)$ denote the assertion of Proposition~\ref{prop:Lpi(q)} for the triple $(p,q,i)$.

\begin{lem}\label{lem:Vpiq_subset}
 Let $p,p',q\ge0$ and $1\le i'<i$.
 \begin{enumerate}
  \item If $p<q-1$ and $p\le p'\le q-1$, then $\tilde{V}_{p',i'}^{(q)}\subset V_{p,i}^{(q)}$.
  \item If $p>q-1$ and $q-1\le p'\le p$, then $\tilde{V}_{p',i'}^{(q)}\subset V_{p,i}^{(q)}$.
  \item If $p>q-1$ and $q\ge1$, then $\tilde{V}_{q-1,i}^{(q)}\subset V_{p,i}^{(q)}$.
 \end{enumerate}
\end{lem}

\begin{proof}
 The lemma can be verified by a straightforward case analysis.
\end{proof}

\begin{lem}\label{lem:induction_differentiation}
 Let $p,q\ge0$ and $i\ge1$.
 Suppose that $P(p',q',i')$ holds whenever $i'<i$ and that $P(p',q',i)$ holds whenever $p'+q'<p+q$.
 \begin{enumerate}
  \item If $p<q-1$ and $i$ is odd, then
   $(1-z)^{-(q-p-1)}L_{(i-1)/2}'\in V_{p,i}^{(q)}$.
  \item If $p<q-1$ and $i$ is even, then
   $z^{-(q-p-1)}L_{i/2}'\in V_{p,i}^{(q)}$.
  \item If $p>q-1$ and $q\ge1$, then
   $z^{p-q+1}L_i'\in V_{p,i}^{(q)}$.
 \end{enumerate}
\end{lem}

\begin{proof}
 In the proof below, we use Lemma~\ref{lem:Vpiq_subset} without further mention.
 \begin{enumerate}
  \item If $i=1$, then the statement is obvious because $L_{(i-1)/2}'=L_0'=0$.
   If $(i-1)/2$ is odd, then $1\le(i-1)/2<i$ and so $P(p,q,(i-1)/2)$ shows that
   \begin{align*}
    (1-z)^{-(q-p-1)}L_{(i-1)/2}'
    &=(1-z)^{-(q-p)}L_{(i-3)/4}\\
    &\in\tilde{V}_{p,(i-1)/2}^{(q)}\subset V_{p,i}^{(q)}.
   \end{align*}
   If $(i-1)/2$ is even and $i\ge5$, then $1\le(i-1)/2<(i+1)/2<i$ and so $P(q-1,q,(i-1)/2)$ and $P(p',q,(i+1)/2)$ for $p'\in\{p+1,\dots,q-1\}$ show that
   \begin{align*}
    (1-z)^{-(q-p-1)}L_{(i-1)/2}'
    &=(1-z)^{-(q-p-1)}z^{-1}L_{(i-1)/4}\\
    &=\Biggl(z^{-1}+\sum_{r=1}^{q-p-1}(1-z)^{-r}\Biggr)L_{(i-1)/4}\\
    &\in\tilde{V}_{q-1,(i-1)/2}^{(q)}+\sum_{p'=p+1}^{q-1}\tilde{V}_{p',(i+1)/2}^{(q)}\subset V_{p,i}^{(q)}.
   \end{align*}
  \item If $i=2$, then $P(q-1,q,1)$ shows that
   \begin{align*}
    z^{-(q-p-1)}L_{i/2}'
    &=z^{-(q-p-1)}L_1'
    =z^{-(q-p-1)}(1-z)^{-1}\\
    &=(1-z)^{-1}+\sum_{r=1}^{q-p-1}z^{-r}\\
    &\in\tilde{V}_{q-1,1}^{(q)}+\langle z^{-r}\mid1\le r\le q-p-1\rangle\subset V_{p,2}^{(q)}=V_{p,i}^{(q)}.
   \end{align*}
   If $i/2$ is even, then $1\le i/2<i$ and so $P(p,q,i/2)$ shows that
   \begin{align*}
    z^{-(q-p-1)}L_{i/2}'
    =z^{-(q-p)}L_{i/4}
    \in\tilde{V}_{p,i/2}^{(q)}\subset V_{p,i}^{(q)}.
   \end{align*}
   If $i/2$ is odd and $i\ge6$, then $1\le(i-2)/2<i/2<i$ and so $P(q-1,q,i/2)$ and $P(p',q,(i-2)/2)$ for $p'\in\{p+1,\dots,q-1\}$ show that
   \begin{align*}
    z^{-(q-p-1)}L_{i/2}'
    &=z^{-(q-p-1)}(1-z)^{-1}L_{(i-2)/4}\\
    &=\Biggl((1-z)^{-1}+\sum_{r=1}^{q-p-1}z^{-r}\Biggr)L_{(i-2)/4}\\
    &\in\tilde{V}_{q-1,i/2}^{(q)}+\sum_{p'=p+1}^{q-1}\tilde{V}_{p',(i-2)/2}^{(q)}\subset V_{p,i}^{(q)}.
   \end{align*}
  \item If $i=1$, then $(q-1)+q<p+q$ and so $P(q-1,q,1)$ shows that
   \begin{align*}
    z^{p-q+1}L_i'
    &=z^{p-q+1}L_1'=z^{p-q+1}(1-z)^{-1}
    =(1-z)^{-1}-\sum_{r=0}^{p-q}z^r\\
    &\in\tilde{V}_{q-1,1}^{(q)}+\langle z^r\mid0\le r\le p-q\rangle\subset V_{p,1}^{(q)}=V_{p,i}^{(q)}.
   \end{align*}
   If $i$ is even, then $1\le i/2<i$ and so $P(p,q,i/2)$ shows that
   \[
    z^{p-q+1}L_i'=z^{p-q}L_{i/2}\in\tilde{V}_{p,i/2}^{(q)}\subset V_{p,i}^{(q)}.
   \]
   If $i$ is odd and $i\ge3$, then $(q-1)+q<p+q$ and $1\le(i-1)/2<i$; therefore $P(q-1,q,i)$ and $P(p',q,(i-1)/2)$ for $p'\in\{q,\dots,p\}$ show that
   \begin{align*}
    z^{p-q+1}L_i'
    &=z^{p-q+1}(1-z)^{-1}L_{(i-1)/2}\\
    &=\Biggl((1-z)^{-1}-\sum_{r=0}^{p-q}z^r\Biggr)L_{(i-1)/2}\\
    &\in\tilde{V}_{q-1,i}^{(q)}+\sum_{p'=q}^{p}\tilde{V}_{p',(i-1)/2}^{(q)}\subset V_{p,i}^{(q)}.\qedhere
   \end{align*}
 \end{enumerate}
\end{proof}

By induction on $(i,p+q)$ with respect to the lexicographic order,
Proposition~\ref{prop:Lpi(q)} reduces to the following lemma:
\begin{lem}
 Let $p,q\ge0$ and $i\ge1$.
 Suppose that $P(p',q',i')$ holds whenever $i'<i$ and that $P(p',q',i)$ holds whenever $p'+q'<p+q$.
 Then $P(p,q,i)$ holds.
\end{lem}

\begin{proof}
 Since $P(p,q,i)$ obviously holds if $q=0$, we may assume that $q\ge1$.
 Then $P(p,q-1,i)$ holds by assumption.

 If $p<q-1$ and $i$ is odd, then $P(p,q-1,i)$ means that
 \[
  L_{p,i}^{(q-1)}\equiv p!(q-p-2)!(1-z)^{-(q-p-1)}L_{(i-1)/2}\pmod{V_{p,i}^{(q-1)}}
 \]
 and so Lemmas~\ref{lem:Vpiq_Vpi(q+1)} and \ref{lem:induction_differentiation}~(1) imply $P(p,q,i)$.

 If $p<q-1$ and $i$ is even, then $P(p,q-1,i)$ means that
 \[
  L_{p,i}^{(q-1)}\equiv(-1)^{q-p-2}p!(q-p-2)!z^{-(q-p-1)}L_{i/2}\pmod{V_{p,i}^{(q-1)}}
 \]
 and so Lemmas~\ref{lem:Vpiq_Vpi(q+1)} and \ref{lem:induction_differentiation}~(2) imply $P(p,q,i)$.

 If $p>q-1$, then $P(p,q-1,i)$ means that
 \[
  L_{p,i}^{(q-1)}\equiv\frac{p!}{(p-q+1)!}L_{p-q+1,i}\pmod{V_{p,i}^{(q-1)}}
 \]
 and so Lemmas~\ref{lem:Vpiq_Vpi(q+1)} and \ref{lem:induction_differentiation}~(3) imply $P(p,q,i)$.

 If $p=q-1$, then $P(p,q-1,i)$ means that
 \[
  L_{p,i}^{(q-1)}\equiv p!L_i\pmod{V_{p,i}^{(q-1)}}
 \]
 and so Lemma~\ref{lem:Vpiq_Vpi(q+1)} shows that
 \[
  L_{p,i}^{(q)}\equiv
  \begin{cases}
   p!(1-z)^{-1}L_{(i-1)/2}&\text{if $i$ is odd};\\
   p!z^{-1}L_{i/2}&\text{if $i$ is even}
  \end{cases}
  \pmod{V_{p,i}^{(q)}},
 \]
 which means that $P(p,q,i)$ holds.
\end{proof}

This completes the proof of Proposition~\ref{prop:Lpi(q)}.

\subsection{Computation of \texorpdfstring{$D_N(\bm{L}_{w,m})$}{$D_N(L_{w,m})$}}
For $w,m\ge1$, we set
\begin{align*}
 \bm{L}_{w,m}
 &=(L_{0,1},\dots,L_{m-1,1},\dots,L_{0,2^w-1},\dots,L_{m-1,2^w-1})\\
 &=(L_{0,i},\dots,L_{m-1,i})_{1\le i\le2^w-1}\\
 &=((L_{p,i})_{0\le p\le m-1})_{1\le i\le2^w-1}\in\Q[[z]]^{(2^w-1)m},
\end{align*}
so that the left-hand side of Proposition~\ref{prop:=} is equal to $D_N(\bm{L}_{w,1})$.
We further set $\bm{L}_{0,m}=\emptyset$ for all $m\ge1$.

\begin{lem}\label{lem:main}
 Let $w,m\ge1$ and $N\ge-1$.
 Then we have
 \[
  D_N(\bm{L}_{w,m}^{(m)})=(-1)^{m(m-1)/2}(\spf_0(m-1))^{2^{w+1}-2}D_{N+2m}(\bm{L}_{w-1,2m}).
 \]
\end{lem}

\begin{proof}
 Notice that for $(p,i)\in\{0,\dots,m-1\}\times\{1,\dots,2^w-1\}$, we have
 \begin{align*}
  V_{p,i}^{(m)}
  &=\langle L_{p',i'}^{(m)}\mid0\le p'\le m-1,\;1\le i'<i\rangle+\langle L_{p',i}^{(m)}\mid0\le p'<p\rangle\\
  &\hphantom{{}={}}+\langle z^{-r}\mid1\le r\le m-p-1\rangle;
 \end{align*}
 the first and second summands on the right-hand side are contained in the $\Q$-linear span of all the entries preceding $L_{p,i}^{(m)}$ in the sequence
 $\bm{L}_{w,m}^{(m)}$.
 Notice also that since $N\ge-1$, the value of $D_N$ depends only on the coefficients of nonnegative powers of $z$ in its arguments.
 Therefore, by elementary column operations from the last column backwards, it follows from Proposition~\ref{prop:Lpi(q)} that
 \[
  D_N(\bm{L}_{w,m}^{(m)})=D_N(\bm{L}_1,\dots,\bm{L}_{2^w-1}),
 \]
 where
 \[
  \bm{L}_i=
   \begin{cases}
    (p!(m-p-1)!(1-z)^{-(m-p)}L_{(i-1)/2})_{0\le p\le m-1}&\text{if $i$ is odd};\\
    ((-1)^{m-p-1}p!(m-p-1)!z^{-(m-p)}L_{i/2})_{0\le p\le m-1}&\text{if $i$ is even}.
   \end{cases}
 \]

 We therefore have
 \begin{align*}
  &D_N(\bm{L}_{w,m}^{(m)})\\
  &=\Biggl(\prod_{p=0}^{m-1}p!(m-p-1)!\Biggr)^{2^{w-1}}\Biggl(\prod_{p=0}^{m-1}(-1)^{m-p-1}p!(m-p-1)!\Biggr)^{2^{w-1}-1}\\
  &\hphantom{{}={}}\times D_N(\tilde{\bm{L}}_1,\dots,\tilde{\bm{L}}_{2^w-1})\\
  &=(-1)^{m(m-1)(2^{w-1}-1)/2}(\spf_0(m-1))^{2^{w+1}-2}D_N(\tilde{\bm{L}}_1,\dots,\tilde{\bm{L}}_{2^w-1}),
 \end{align*}
 where
 \[
  \tilde{\bm{L}}_i=
   \begin{cases}
    ((1-z)^{-(m-p)}L_{(i-1)/2})_{0\le p\le m-1}&\text{if $i$ is odd};\\
    (z^{-(m-p)}L_{i/2})_{0\le p\le m-1}&\text{if $i$ is even}.
   \end{cases}
 \]

 Since
 \[
  \tilde{\bm{L}}_1=((1-z)^{-(m-p)})_{0\le p\le m-1}=((1-z)^{-m},\dots,(1-z)^{-1}),
 \]
 Lemma~\ref{lem:1-z} shows that
 \begin{align*}
  &D_N(\tilde{\bm{L}}_1,\dots,\tilde{\bm{L}}_{2^w-1})\\
  &=(-1)^{m(m-1)/2}D_{N+m}((1-z)^m\tilde{\bm{L}}_2,\dots,(1-z)^m\tilde{\bm{L}}_{2^w-1}).
 \end{align*}
 Here we have
 \begin{align*}
  &D_{N+m}((1-z)^m\tilde{\bm{L}}_2,\dots,(1-z)^m\tilde{\bm{L}}_{2^w-1})\\
  &=D_{N+2m}(z^m(1-z)^m\tilde{\bm{L}}_2,\dots,z^m(1-z)^m\tilde{\bm{L}}_{2^w-1})\\
  &=D_{N+2m}(\bm{K}_1,\dots,\bm{K}_{2^{w-1}-1})
 \end{align*}
 where
 \begin{align*}
  \bm{K}_i&=(z^m(1-z)^m\tilde{\bm{L}}_{2i},z^m(1-z)^m\tilde{\bm{L}}_{2i+1})\\
  &=((z^p(1-z)^mL_i)_{0\le p\le m-1},(z^m(1-z)^pL_i)_{0\le p\le m-1})\in\Q[[z]]^{2m}.
 \end{align*}
 Performing elementary column operations gives
 \begin{align*}
  &D_{N+2m}(\bm{K}_1,\dots,\bm{K}_{2^{w-1}-1})\\
  &=D_{N+2m}(((z^p(1-z)^mL_i)_{0\le p\le m-1},(z^m(1-z)^pL_i)_{0\le p\le m-1})_{1\le i\le 2^{w-1}-1})\\
  &=D_{N+2m}(((z^p(1-z)^mL_i)_{0\le p\le m-1},((-1)^pz^{m+p}L_i)_{0\le p\le m-1})_{1\le i\le 2^{w-1}-1})\\
  &=D_{N+2m}(((z^pL_i)_{0\le p\le m-1},((-1)^pz^{m+p}L_i)_{0\le p\le m-1})_{1\le i\le 2^{w-1}-1})\\
  &=(-1)^{m(m-1)(2^{w-1}-1)/2}D_{N+2m}(\bm{L}_{w-1,2m}).
 \end{align*}

 Hence we conclude that
 \begin{align*}
  &D_N(\bm{L}_{w,m}^{(m)})\\
  &=(-1)^{m(m-1)(2^{w-1}-1)/2}(\spf_0(m-1))^{2^{w+1}-2}(-1)^{m(m-1)/2}\\
  &\hphantom{{}={}}\times(-1)^{m(m-1)(2^{w-1}-1)/2}D_{N+2m}(\bm{L}_{w-1,2m})\\
  &=(-1)^{m(m-1)/2}(\spf_0(m-1))^{2^{w+1}-2}D_{N+2m}(\bm{L}_{w-1,2m}),
 \end{align*}
 as desired.
\end{proof}

\begin{lem}\label{lem:recursion_L_w,m}
 Let $w,m\ge1$ and $N\ge m-1$.
 Then we have
 \[
  D_N(\bm{L}_{w,m})=(-1)^{m(m-1)/2}\frac{(\spf_0(m-1))^{2^{w+1}-2}\spf_{N-m}(m)}{\spf_{N+(2^w-2)m}(m)}D_{N+m}(\bm{L}_{w-1,2m}).
 \]
\end{lem}

\begin{proof}
 Lemmas~\ref{lem:diff} and \ref{lem:main} show that
 \begin{align*}
  D_N(\bm{L}_{w,m})&=\frac{\spf_{N-m}(m)}{\spf_{N+(2^w-2)m}(m)}D_{N-m}(\bm{L}_{w,m}^{(m)})\\
  &=(-1)^{m(m-1)/2}\frac{(\spf_0(m-1))^{2^{w+1}-2}\spf_{N-m}(m)}{\spf_{N+(2^w-2)m}(m)}D_{N+m}(\bm{L}_{w-1,2m}),
 \end{align*}
 as desired.
\end{proof}

\begin{rem}\label{rem:D_N(L_1,m)}
 Since $D_N(\emptyset)=1$, applying Lemma~\ref{lem:recursion_L_w,m} to $w=1$ gives
 \begin{align*}
  D_N(\bm{L}_{1,m})
  &=D_N(L_{0,1},\dots,L_{m-1,1})=\begin{vmatrix}\frac{1}{N+1}&\dots&\frac{1}{N-m+2}\\\vdots&\ddots&\vdots\\\frac{1}{N+m}&\dots&\frac{1}{N+1}\end{vmatrix}\\
  &=(-1)^{m(m-1)/2}\frac{(\spf_0(m-1))^2\spf_{N-m}(m)}{\spf_{N}(m)}.
 \end{align*}
 This identity also follows from the formula for the Cauchy determinant
 \[
  \det\biggl(\frac{1}{x_i-y_j}\biggr)_{1\le i,j\le m}=\frac{\prod_{1\le i<j\le m}(x_j-x_i)(y_i-y_j)}{\prod_{i,j=1}^{m}(x_i-y_j)}
 \]
 by setting $x_i=N+i+1$ and $y_i=i$:
 \begin{align*}
  D_N(\bm{L}_{1,m})
  &=\begin{vmatrix}\frac{1}{N+1}&\dots&\frac{1}{N-m+2}\\\vdots&\ddots&\vdots\\\frac{1}{N+m}&\dots&\frac{1}{N+1}\end{vmatrix}\\
  &=\frac{\prod_{1\le i<j\le m}((N+j+1)-(N+i+1))(i-j)}{\prod_{i,j=1}^{m}(N+i+1-j)}\\
  &=(-1)^{m(m-1)/2}\frac{\prod_{1\le i<j\le m}(j-i)^2}{\prod_{i=1}^{m}((N+i)!/(N+i-m)!)}\\
  &=(-1)^{m(m-1)/2}\frac{(\spf_0(m-1))^2\spf_{N-m}(m)}{\spf_{N}(m)}.
 \end{align*}
\end{rem}

\begin{thm}\label{thm:D_N(L_w,m)}
 For $w,m\ge1$ and $N\ge m-1$, we have
 \begin{align*}
  D_N(\bm{L}_{w,m})
  &=(-1)^{m(m+1)/2-m\delta_{w,1}}\\
  &\hphantom{{}={}}\times\prod_{j=0}^{w-1}\prod_{i=\lfloor2^{j-1}\rfloor m}^{2^jm-1}\frac{\bigl(i!\bigr)^{2^{w-j+2}-2(w-j+2)}((N-m+1+i)!)^{w-j}}{((N+(2^w-1)m-i)!)^{w-j}}.
 \end{align*}
\end{thm}

\begin{proof}
 We prove by induction on $w$ that the formula holds for all $m\ge1$ and all $N\ge m-1$.
 For $w=1$, the formula follows from Lemma~\ref{lem:recursion_L_w,m}, as explained in Remark~\ref{rem:D_N(L_1,m)}.
 Let $w\ge2$ and suppose that the formula holds for $w-1$ for all $m\ge1$ and all $N\ge m-1$.
 Since $N+m\ge2m-1$, Lemma~\ref{lem:recursion_L_w,m} and the induction hypothesis show that
 \begin{align*}
  &D_N(\bm{L}_{w,m})\\
  &=(-1)^{m(m-1)/2}\frac{(\spf_0(m-1))^{2^{w+1}-2}\spf_{N-m}(m)}{\spf_{N+(2^w-2)m}(m)}\\
  &\hphantom{{}={}}\times(-1)^{2m(2m+1)/2-2m\delta_{w-1,1}}\\
  &\hphantom{{}={}}\times\prod_{j=0}^{w-2}\prod_{i=\lfloor2^{j-1}\rfloor2m}^{2^{j+1}m-1}\frac{(i!)^{2^{w-j+1}-2(w-j+1)}((N-m+1+i)!)^{w-j-1}}{((N+m+(2^{w-1}-1)2m-i)!)^{w-j-1}}\\
  &=(-1)^{m(m+1)/2}\prod_{i=0}^{m-1}\frac{(i!)^{2^{w+1}-2}(N-m+1+i)!}{(N+(2^w-1)m-i)!}\\
  &\hphantom{{}={}}\times\prod_{i=0}^{2m-1}\frac{(i!)^{2^{w+1}-2(w+1)}((N-m+1+i)!)^{w-1}}{((N+(2^w-1)m-i)!)^{w-1}}\\
  &\hphantom{{}={}}\times\prod_{j=1}^{w-2}\prod_{i=\lfloor2^{j-1}\rfloor2m}^{2^{j+1}m-1}\frac{(i!)^{2^{w-j+1}-2(w-j+1)}((N-m+1+i)!)^{w-j-1}}{((N+(2^w-1)m-i)!)^{w-j-1}}\\
  &=(-1)^{m(m+1)/2}\prod_{i=0}^{m-1}\frac{(i!)^{2^{w+2}-2(w+2)}((N-m+1+i)!)^w}{((N+(2^w-1)m-i)!)^w}\\
  &\hphantom{{}={}}\times\prod_{i=m}^{2m-1}\frac{(i!)^{2^{w+1}-2(w+1)}((N-m+1+i)!)^{w-1}}{((N+(2^w-1)m-i)!)^{w-1}}\\
  &\hphantom{{}={}}\times\prod_{j=2}^{w-1}\prod_{i=2^{j-1}m}^{2^jm-1}\frac{(i!)^{2^{w-j+2}-2(w-j+2)}((N-m+1+i)!)^{w-j}}{((N+(2^w-1)m-i)!)^{w-j}}\\
  &=(-1)^{m(m+1)/2}\prod_{j=0}^{w-1}\prod_{i=\lfloor2^{j-1}\rfloor m}^{2^jm-1}\frac{(i!)^{2^{w-j+2}-2(w-j+2)}((N-m+1+i)!)^{w-j}}{((N+(2^w-1)m-i)!)^{w-j}},
 \end{align*}
 as required.
\end{proof}

\begin{proof}[Proof of Proposition~\ref{prop:=}]
 Set $m=1$ in Theorem~\ref{thm:D_N(L_w,m)} to find the value of $D_N(\bm{L}_{w,1})$.
\end{proof}

\section*{Acknowledgements}
This work was supported by JSPS KAKENHI Grant Numbers JP22K13897, JP22K03244, and JP23K03072.

The authors used ChatGPT (OpenAI) and Claude (Anthropic) to proofread the manuscript and to improve its exposition.
All suggestions were reviewed by the authors, who take full responsibility for the content of this paper.
The third author thanks OpenAI for providing complimentary access to a ChatGPT workspace through the ChatGPT for Academic Researchers programme.
The second and third authors gratefully acknowledge Anthropic for providing complimentary access to Claude through the Claude Team plan for scientists.

\end{document}